\documentclass[a4paper,12pt]{amsproc}

\usepackage[utf8]{inputenc}

\usepackage{amsmath}
\usepackage{amssymb}
\usepackage{mathrsfs}
\usepackage[dvips]{graphicx}

\newtheorem{theorem}{Theorem}[section]
\newtheorem{cor}[theorem]{Corollary} 
\newtheorem{definition}[theorem]{Definition}

\def\bbn{\mathbb{N}}

\def\bbr{\mathbb{R}} 

\def\w{\omega}

\def\ca{\mathcal{A}}
\def\cb{\mathcal{B}}

\def\cd{\mathcal{D}}

\def\ci{\mathcal{I}}
\def\cm{\mathcal{M}}
\def\cn{\mathcal{N}}
\def\cp{\mathcal{P}}

\def\b{\mathfrak{b}}
\def\c{\mathfrak{c}}
\def\d{\mathfrak{d}}

\DeclareMathOperator{\Perf}{Perf}

\DeclareMathOperator{\Borel}{Borel}
\DeclareMathOperator{\Laver}{LaverTrees}
\DeclareMathOperator{\Miller}{MillerTrees}

\DeclareMathOperator{\ZFC}{ZFC}

\DeclareMathOperator{\add}{add}
\DeclareMathOperator{\non}{non}
\DeclareMathOperator{\cov}{cov}
\DeclareMathOperator{\cof}{cof}

\def\ad{{ a.d.\,}}
\def\mad{{ m.a.d.\,}}
\def\ed{{ e.d.\,}}

\def\restricted{\upharpoonright}

\def\then{\longrightarrow}

\def\then{\longrightarrow}

\def\<{\langle}
\def\>{\rangle}

\def\Repicky{{Repick{\'y} }}

\title[Nonmeasurable sets and unions]{Nonmeasurable sets and unions with respect to selected ideals especially ideals defined by trees}
\author{Robert Rałowski and Szymon Żeberski}

\address{ Department of Computer Science, Faculty of Fundamental Problems of Technology, Wroc\l aw University of Technology, Wybrzeże Wyspiańskiego 27, 50-370 Wroc\l aw, Poland}

\thanks{\hspace*{-0.77cm} AMS Classification: Primary 03E17, 03E50, 03E75; Secondary 28A99\\
Keywords: Marczewski ideal, Laver tree, Miller tree, m.a.d. family, dominating family, nonmeasurable set, completely nonmeasurable set, Bernstein set, Polish space, Continuum Hypothesis.\\
The work has been partially financed by grant S40012/K1102 from the Faculty of Fundamental Problems of Technology of Wrocław University of Technology.}

\email[Robert Rałowski]{robert.ralowski@pwr.edu.pl}
\email[Szymon Żeberski]{szymon.zeberski@pwr.edu.pl}

\begin{document}

\begin{abstract} In this paper we consider nonmeasurablity with respect to $\sigma$-ideals defined be trees. First classical example of such ideal is Marczewski ideal $s_0.$ We will consider also ideal $l_0$ defined by Laver trees and $m_0$ defined by Miller trees. With the mentioned ideals one can consider $s$, $l$ and $m$-measurablility. 

We have shown that there exists a subset $A$ of the Baire space which is $s$, $l$ and $m$ nonmeasurable at the same time. Moreover, $A$ forms m.a.d. family which is also dominating. We show some examples of subsets of the Baire space which are measurable in one sense and nonmeasurable in the other meaning.

We also examine terms nonmeasurable and completely nonmeasurable (with respect to several ideals with Borel base). There are several papers about finding (completely) nonmeasurable sets which are the union of some family of small sets. In this paper we want to focus on the following problem:
"Let $\cp$ be a family of small sets. Is it possible that
for all $\ca\subseteq\cp$,
$\bigcup\ca$  is nonmeasurable  implies that
$\bigcup\ca$ is completely nonmeasurable?"

We will consider situations when $\cp$ is a partition of $\bbr$, $\cp$ is point-finite family and $\cp$ is point-countable family. We give an equivalent statement to CH using terms nonmeasurable and completely nonmeasurable.

\end{abstract}

\maketitle

\section{Notation}
We will use standard set-theoretic notation following e.g. \cite{Jech}.  
$\bbr$ will denote the real line. For a set $X$, $P(X)$ denotes the power set of $X$ and $|X|$ denotes the cardinality of $X$. If $\kappa$ is a cardinal number then
$$\begin{array}{r@{\; =\; }l}
  [X]^\kappa & \{A\subseteq X:\ |A|=\kappa\}\\  
  {[X]^{<\kappa}} & \{A\subseteq X:\ |A|<\kappa\}\\  
  {[X]^{\le\kappa}} & \{A\subseteq X:\ |A|\le\kappa\}\\
  \end{array}
$$

Let $X$ be any uncountable Polish space with $\ci$ an arbitrary $\sigma$-ideal on $\cp(X)$ and let us recall the cardinal coefficients of $\ci$
\begin{itemize}
\item $\non(\ci)=\min \{ |F|:\; F\subseteq X\land F\notin \ci\},$
\item $\add(\ci)=\min \{ |\ca|:\; \ca\subseteq \ci\land \bigcup\ca\notin \ci \},$
\item $\cof(\ci)=\min \{ |\cb|:\; \cb\subseteq \ci\land (\forall A\in \ci)(\exists B\in\cb) A\subseteq B\},$
\item $\cov(\ci)=\min \{ |\ca|:\; \ca\subseteq \ci\land \bigcup\ca = X\},$
\item for a fixed family of perfect subsets $\cp \subseteq Perf(X)$ let\\ $\cov_h(\ci)=\min \{ |\ca|:\; \ca\subseteq \ci\land (\exists P\in \cp)\; P\subseteq \bigcup\ca\}$.
\end{itemize}
Let us recall the definition of a bounding number.
$$ 
 \b = \min \{| \cb |:\ \cb \subseteq \w^\w \land (\forall x \in \w^\w) (\exists y \in \cb) \; \neg (s \le^* x) \} 
 $$ 

In \cite{Marczewski} Marczewski introduced the notion of $s$-measurability and the $s_0$-ideal. Recalling these definitions we have:
\begin{definition}[Marczewski ideal $s_0$] Let $X$ be any fixed uncountable Polish space. Then we say that $A\in \cp(X)$ is in $s_0$ iff
$$
(\forall P\in Perf(X))(\exists Q\in Perf(X))\; Q\subseteq P\land Q\cap A=\emptyset.
$$
\end{definition}
Notice that for this ideal we have $\cov(s_0)=\cov_h(s_0)$ and this cardinal is the same for all uncountable Polish spaces. 
To see this use the fact that in any uncountable Polish space there is a disjoint maximal antichain $\ca$ (of cardinality $\c$) consisting of Cantor perfect sets. 
From this it follows that $B\in s_0$ if and only if $(\forall A\in\ca)\; B\cap A\in s_0$.

\begin{definition}[$s$-measurable set] Let $X$ be any fixed uncountable Polish space. Then we say that $A\in \cp(X)$ is {\bf $s$-measurable} iff
$$
(\forall P\in \Perf(X))(\exists Q\in \Perf(X))\; Q\subseteq P\land (Q\subseteq A\lor Q\cap A=\emptyset).
$$
Moreover, a set $A\in \cp(X)$ is a {\bf Bernstein set} if
$$
(\forall P\in Perf(X))\; P\cap A\ne\emptyset\land P\cap A^c\ne \emptyset,
$$
(where $A^c$ denotes complement of the set A in space $X$).
\end{definition}

\begin{definition} Let $X$ be any uncountable Polish space and let us consider a cardinal $\kappa$. We say that the family 
$\ca\subseteq P(X)$ is {\bf $\kappa$-point family} iff $|\{ A\in\ca:\ x\in A\}|<\kappa$ for all $x\in X$. 
\end{definition}
We say that $\ca$ is point-finite family if $\ca$ is $\omega$-point family and  $\ca$ is countable-point family if $\ca$ is $\omega_1$-point family.

We say that $\sigma$-ideal $\ci$ of subsets of some Polish space $X$ has Borel base if for any set $A\in \ci$ there is a Borel set $B\in Bor(X)\cap\ci$ such that $A\subseteq B$. 
Classical examples of ideals possesing Borel base on the real line are 
\begin{itemize}
 \item the $\sigma$-ideal $\ci=[\mathbb{R}]^{\le\omega}$ of all countable subsets,
 \item the $\sigma$-ideal $\cm$ of meager subsets, 
 \item the $\sigma$-ideal $\cn$ of null subsets with respect to Lebesgue measure.
\end{itemize}

For fixed $\sigma$-ideal $\ci$ with Borel base we say that a subset $A\subseteq X$ of Polish space $X$ is measurable with respect to $\ci$ iff $A$ belongs to $\sigma$-algebra $Bor[\ci]$ generated by Borel subsets of $X$ and $\sigma$-ideal $\ci$. 

In the first part of this paper we consider subsets connected to $\sigma$-ideal without Borel base generated by trees. We are interested in measurability connected to Laver trees and Miller trees and it's interplay with m.a.d. families.

In the second part we investigate subsets connected to $\sigma$-ideals with Borel base.
We discuss the difference between measurability and complete nonmeasurability of unions of small sets.

\section{m.a.d. families and their $s$, $l$ and $m$-measurability}

For every tree $T\subseteq \omega^{<\omega}$ let $[T]$ be the set of all branches of $T$ which is defined as follows:
$$
[T]=\{ x\in \omega^\omega:\; (\forall n\in\omega)\; x\restricted n\in T\}.
$$

We say that a tree $T\subseteq \omega^{<\omega}$ is called a {\bf Laver tree} iff there is a node $s\in T$ such that, 
for every node $t\in T$ if $s\subseteq t$ then $t$ is infinitely spliting i.e. $\{ n\in\omega:\; s^\frown n\in T\}$ is infinite.

The set of all Laver trees is denoted by the $\Laver$. Moreover, recalling th definition of the ideal $l_0$, we have
\begin{definition}[ideal $l_0$] We say that $A\in \cp(\omega^\omega)$ is in $l_0$ iff
$$
(\forall T\in \Laver)(\exists Q\in \Laver)\; Q\subseteq T\land [Q]\cap A=\emptyset.
$$
\end{definition}
\begin{definition}[$l$-measurable set] We say that $A\in \cp(\omega^\omega)$ is {\bf $l$-measurable} iff for every Laver tree $T\in \Laver$ there is a Laver tree $S\in\Laver$ such that
$$
(S\subseteq T\land [S]\subseteq A)\lor (S\subseteq T\land [S]\cap A=\emptyset).
$$
\end{definition}

We say that a tree $T\subseteq \omega^{<\omega}$ is called a {\bf Miller tree} iff there is a node $s\in T$ such that, 
for every node $t\in T$ if $s\subseteq t$ then there is $t'$ such that $t\subseteq t'$ and $t'$ is infinitely spliting. 

The set of all Miller trees is denoted by the $\Miller$. Moreover, recalling th definition of the ideal $m_0$, we have
\begin{definition}[ideal $m_0$] We say that $A\in \cp(\omega^\omega)$ is in $m_0$ iff
$$
(\forall T\in \Miller)(\exists Q\in \Miller)\; Q\subseteq T\land [Q]\cap A=\emptyset.
$$
\end{definition}
\begin{definition}[$m$-measurable set] We say that $A\in \cp(\omega^\omega)$ is {\bf $m$-measurable} iff for every Miller tree $T\in \Miller$ there is a Miller tree $S\in\Miller$ such that
$$
(S\subseteq T\land [S]\subseteq A)\lor (S\subseteq T\land [S]\cap A=\emptyset).
$$
\end{definition}

It is well known by Judah, Miller, Shelah see \cite{JMS} and \Repicky see \cite{Rep} that $add(s_0)\le cov(s_0)\le cof(\c)\le non(s_0)=\c<cof(s_0)\le 2^\c$.
Cleary $\omega_1\le \add(l_0)\le \cov(l_0)\le \c$ holds. 
Moreover, in \cite{Goldstern}, Goldstern, \Repicky\!\!, Shelah and Spinas showed that it is relatively consistent with $\ZFC$ that $\add(l_0)<\cov(l_0)$.


Using the natural isomorphism of the perfect set $P=[T]$, where $T$ is Laver tree, with $\omega^\omega$, we have $cov_h(l_0)=cov(l_0)$.


Let us recall the definition of almost disjoint family. Any family of sets $\ca\subseteq [\omega]^\omega$ is an {\bf \ad--family} on $\omega$ if
$$
(\forall a,b\in\ca)\; a\ne b\then a\cap b\in [\omega]^{<\omega}.
$$
Two reals $f,g\in\omega^\omega$ in Baire space are {\bf eventually different \ed} iff $f\cap g$ is a finite subset of $\omega\times\omega$. 
Let us observe that an \ed family $\ca\subseteq\omega^\omega$ is an \ad family on $\omega\times\omega$. 
For this reason we will identify the notions of eventually different family and almost disjoint \ad family. 
Maximal almost disjoint (or eventually different) families with respect to inclusion are called {\bf \mad families}.

Every \ad family is a meager subset of the Cantor space and every $\ed$ family is a meager subset of the Baire space. 
It is natural to ask whether the existence \mad families that are either $s$-measurable or $s$-nonmeasurable can be proven in $\ZFC$ alone. 
One can find a consistent example of a \mad family $\ca$ of cardinality smaller than $\c$ (see \cite{Kunen}, for example). 
In a this case $\ca\in s_0$ since we have $\non(s_0)=\c$. Furthermore it is well known that there exists a perfect \ad family and therefore not all \mad families are in $s_0$.

\begin{theorem}\label{nonmeasurable_mad} There exists a $s$-nonmeasurable  m.a.d. family in Baire space.
\end{theorem}
\begin{proof} Fix $T\subseteq \omega^{}<\omega$ a perfect tree such that $[T]$ is a.d. in $\omega^\omega$. Let us enumerate $Perf(T)=\{ T_\alpha: \alpha<\c\}$ a family of all perfect subsets of $T$. 
By transfinite reccursion let us define
$$
\{ (a_\alpha,d_\alpha,x_\alpha)\in [T]^2\times \omega^\omega:\alpha<\c\}
$$
such that for any $\alpha<\c$ we have:
\begin{enumerate}
 \item $a_\alpha,d_\alpha\in T_\alpha$,
 \item $\{ a_\xi:\xi<\alpha\}\cap \{ d_\xi:\xi<\alpha\}=\emptyset$,
 \item $\{ a_\xi:\xi<\alpha\}\cup \{ x_\xi:\xi<\alpha\}$ is a.d.,
 \item $\forall^\infty n\; x_\alpha(n)=d_\alpha(n)$ but $x_\alpha\ne d_\alpha$.
\end{enumerate}
Now assume that we are in $\alpha$-th step construction and we have required sequence
$$
\{ (a_\xi,d_\xi,x_\xi)\in [T]^2\times \omega^\omega:\xi<\alpha\}
$$
which have size at most $\omega|\alpha|<\c$ then we can choose in $[T_\alpha]$ (of size $\c$) $a_\alpha,d_\alpha\in [T_\alpha]$ which fulfills the first condition. 
Then choose any $x_\alpha\in \omega^\omega$ different than $d_\alpha$ but $(\forall^\infty n) d_\alpha(n)=x_\alpha(n)$ then $x\in\omega^\omega\setminus [T]$ and 
$$
\{ a_\xi:\xi<\alpha\}\cup \{ x_\xi:\xi<\alpha\}
$$
forms an a.d. family in $\omega^\omega$. Then $\alpha$-th step construction is completed. By transfinite induction theorem we have required sequence of the length $\c$. 
Now set $A_0 = \{ a_\alpha:\alpha < \c\}\cup \{ x_\alpha:\alpha<\c\}$ and let us extend it to any maximal a.d. family $A$. It is easy to chect that $A$ is required $s$-nonmeasurable m.a.d. family in the Baire space $\omega^\omega$.
\end{proof}

The next theorem generalizes result obtained in \cite{R1}.
\begin{theorem}\label{d_l_mad} There exists a \mad family of functions 
$\ca\subseteq \omega^\omega$ such that $\ca$ is not $s,l,m$-measurable at the same time, and there is an dominating subfamily $\ca'\in [\ca]^{\le \d}$ in Baire space $\omega^\omega$.
\end{theorem}
\begin{proof} Now by assumption there is a dominating family $\cd_0\subseteq \omega^\omega$ of size $\d$. Then we show the existence an a.d. dominating family $\cd$ of the same size. 
To do let $\cp = \{ A_m\in [\omega]^\omega:\; m\in\omega \}$ be a partition of $\omega$ ont infinite subsets. Now let us construct a tree as follows: 
$T_{-1}=\{ \emptyset \}$, next $T_0=\{ (0,n): n\in\omega\}$. 
Now assume that we have defined $T_n$ for a fixed $n\in\omega$ and let us enumerate $T_n=\{ s_k: k\in\omega\}$ then for any $m\in\omega$ let us set $A_m={ k_i\in\omega: k\in\omega}$ as an increasing sequence and define 
$T_{n+1,m} = \{ s_m\cup \{(n+1, k_i)\}:i\in\omega\}$ and then let $T_{n+1}=\bigcup_{m\in\omega} T_{n+1,m}$ and finally $T=\bigcup_{n\in\omega\cup\{ -1\}} T_n$.  
It is easy to observe that $[T]$ forms a a.d. family of reals in $\omega^\omega$. 

Now let us define an embedding $F:\cd_0\to [T]$ as follows: pick an arbitrary element $d\in\cd_0$ which is an union $\bigcup\{ d\upharpoonright n: n\in\omega\}$ 
then assign to $d\restricted 0=\emptyset\in T_{-1}$ and to $d\restricted 1$ $t_0=d\restricted 1=\{ (0,d(0))\}$. Now let us assume that we have assigned for a fixed 
$d\restricted n$  $t_n\in T_n$ for $n\in\omega$. Then there is unique $m\in\omega$ such that $t_n\in T_{n,m}$ but $A_m=\{ k_i:i\in\omega\}$ is represented by the increasing sequence 
$(k_i)_{i\in\omega}\in\omega^\omega$ then $d\restricted n+1$ is assigned to $t_{n+1} = t_n\cup \{ (n+1,w)\}$ where $w=k_{d(n+1)}$ which is a greater than $d(n+1)$ of course. 
From the construction we see that $t_{n+1}\in T_{n+1}$ and for any $n\in\omega$ $t_n\subseteq t_{n+1}$. Now let $f(d)=\bigcup\{ t_n\in T_n: n\in\omega:\}\in [T]$. 
It easy to see that this construction ensure that $f$ is one to one mapping and for any $d\in \cd_0$ $d\le f(d)$. Now let $\cd=\{ 4\cdot f(d): d\in\cd_0\}\subseteq (4\bbn)^\omega$ 
which forms a dominating family in $\omega^\omega$ of size equal to $\d=|\cd_0|$.

Now let us choose a.d. trees $S\subseteq (4\bbn+1)^{<\omega}$, $M\subseteq (4\bbn+2)^{<\omega}$ and $L\subseteq (4\bbn+3)^{<\omega}$ where $S$ is a perfect tree, $M$ is Miller 
and last $L$ is a Laver tree.

Let us enumerate $Perf(S)=\{ T_\alpha: \alpha<\c\}$ a family of all perfect subsets of $S$ and analogously $Miller(M)=\{ M_\alpha:\alpha<\c\}$, $Laver(L)=\{ L_\alpha:\alpha<\c\}$. 
By transfinite reccursion let us define
$$
\{ w_\alpha \in [S]^2\times \omega^\omega\times [M]^2\times \omega^\omega\times [L]^2\times \omega^\omega:\alpha<\c\}
$$
where $w_\alpha=(a^s_\xi,d^s_\xi,x^s_\xi,a^s_\xi,d^s_\xi,x^s_\xi,a^s_\xi,d^s_\xi,x^s_\xi,)$ for any $\alpha<\c$, and such that for any $\alpha<\c$ we have:
\begin{enumerate}
 \item $a^s_\alpha,d^s_\alpha\in S_\alpha$,
 \item $\{ a^s_\xi:\xi<\alpha\}\cap \{ d^s_\xi:\xi<\alpha\}=\emptyset$,
 \item $\{ a^s_\xi:\xi<\alpha\}\cup \{ x^s_\xi:\xi<\alpha\}$ is a.d.,
 \item $\forall^\infty n\; x^s_\alpha(n)=d^s_\alpha(n)$ but $x^s_\alpha\ne d^s_\alpha$.
 
 \item $a^m_\alpha,d^m_\alpha\in M_\alpha$,
 \item $\{ a^m_\xi:\xi<\alpha\}\cap \{ d^m_\xi:\xi<\alpha\}=\emptyset$,
 \item $\{ a^m_\xi:\xi<\alpha\}\cup \{ x^m_\xi:\xi<\alpha\}$ is a.d.,
 \item $\forall^\infty n\; x^m_\alpha(n)=d^m_\alpha(n)$ but $x^m_\alpha\ne d^m_\alpha$.

 \item $a^l_\alpha,d^l_\alpha\in L_\alpha$,
 \item $\{ a^l_\xi:\xi<\alpha\}\cap \{ d^l_\xi:\xi<\alpha\}=\emptyset$,
 \item $\{ a^l_\xi:\xi<\alpha\}\cup \{ x^l_\xi:\xi<\alpha\}$ is a.d.,
 \item $\forall^\infty n\; x^l_\alpha(n)=d^l_\alpha(n)$ but $x^l_\alpha\ne d^l_\alpha$.
\end{enumerate}
Now assume that we are in $\alpha$-th step construction and we have required sequence
$$
\{ w_\alpha :\; \xi<\alpha\}
$$
which have size at most $\omega\cdot |\alpha|<\c$. In case of perfect part we can choose in $[S_\alpha]$ (of size $\c$) $a^s_\alpha,d^s_\alpha\in [S_\alpha]$ which fulfills the first condition. 
Then choose any $x^s_\alpha\in \omega^\omega$ different than $d^s_\alpha$ but $(\forall^\infty n) d_\alpha(n)=x_\alpha(n)$ then $x^s_\alpha\in\omega^\omega\setminus [S]$ and 
$$
\{ a_\xi:\xi \le \alpha\}\cup \{ x_\xi:\xi \le \alpha\}
$$
forms an a.d. family in $\omega^\omega$. In the same way we can choose rest points of our tuple for Miller and Laver trees. Then $\alpha$-th step construction is completed. 
By transfinite induction theorem we have required sequence of the length $\c$. Now set 
$$
A_s = \cd\cup \{ a^s_\alpha:\alpha < \c\}\cup \{ x^s_\alpha:\alpha<\c\},
$$
$$
A_m = \cd\cup \{ a^m_\alpha:\alpha < \c\}\cup \{ x^m_\alpha:\alpha<\c\}
$$
and
$$
A_l = \cd\cup \{ a^l_\alpha:\alpha < \c\}\cup \{ x^l_\alpha:\alpha<\c\}
$$
and let us extend the family $\ca=\cd\cup\ca_s\cup\ca_m\cup\ca_l$ to any maximal a.d. family $A$. It is easy to check that $A$ is required $s,m$ 
and $l$-nonmeasurable m.a.d. family in the Baire space $\omega^\omega$ with a dominating subfamily of size $\cd$, what completes this proof.
\end{proof}

Now, let us give some examples of subsets of Baire space which are nonmeasurable in one sense and measurable in other one at the same time. Similar results were obtained in \cite{Brendle}.
\begin{theorem} There are subsets $A,B,C$ of the Baire space such that
\begin{itemize}
 \item $A$ is $l$-measurable and not $s$-measurable,
 \item $B$ is $m$-measurable  but not $s$-measurable,
 \item $C$ is $l$-measurable but not $m$-measurable.
\end{itemize}
Moreover, if $\b=\c$ then 
\begin{itemize}
 \item there is a not $l$-measurable set which is $s$-measurable
 \item there is a not $m$-measurable set which is $s$-measurable.
\end{itemize}
\end{theorem}
\begin{proof} To show the first part, let $T$ be $a.d.$-disjoint Laver tree (defined as in the proof of Theorem \ref{d_l_mad}). Now let us define a perfect subtree $S$ of $T$ such that all levels of $S$ consist the first two numbers of $T$, i.e. $\tau\in S$ if $\tau\in T$ and for any $i\in dom(\tau)$ $\tau(i)\in\{m,n\}$ where $m=\min\{ s(i):s\in T\}$ and $n=\min\{ s(i): s\in T\land s(i)\ne m\}$. Now choose any $X\subset [S]$ which is not $s$-measurable and then let $A=([T]\setminus [S])\cup X$. Now let us observe that $[S]\in l_0$ and then $X\in l_0$ what give the assertion.
                                                                                                                                                                                                                                                                                                                                                                                                                                                                                                   
The similar argument shews the second clause.

To see next, let $T$ be as above and let $M\subseteq T$ be a Miller tree defined as follows: on the odd levels of $T$ each node of $M$ does not split but in evan levles of $T$ the all nodes of $M$ uses half part of the level of $T$ i.e. if $s\in M$ and $i\in dom(s)$ is even then $s(i)\in \{a_{2k}:k\in\omega\}$ where $\{ a_k:k\in\omega\}$ enumeration of 
$$
\{ m\in\omega: (\exists \tau\in T)\; \tau=s\restricted i\land \tau^\frown m\in T\}.
$$
Firstly we show that $[M]\in l_0$, let us consider any Laver subtree $S\subseteq T$ with a stem $s\in S$ then let us find a node $\tau\in S$ which extend $s$ and $|\tau|$ is odd but every node $\tau'\in M$ with $|\tau'|$ is odd has no splitting one then we can find infinite set of splitting i.e. 
$$
W_\tau=\{ n\in\omega:\; \tau^\frown \in S\setminus M\}
$$
and then we can a Laver subtree $S'$ of $S$ such that $[S']\cap [M]$ is empty and then we have conclusion. Now let us consider any not $m$-measurable subset $Y\subset [M]$ then $B=([T]\setminus [M])\cup Y$ is as we want.

To prove the next sentence let us enumerate all Laver subtrees of $T$ $T_\alpha:\alpha<\c\}$ and all perfect subtrees $\{ S_\alpha<\c]|$ of $T$. Then let us define a transfinite sequence:
$$
(a_\xi,d_\xi,P_\xi):\xi<\c
$$
with the following conditions, for any $\xi<\c$
\begin{enumerate}
 \item $a_\xi,d_\xi\in [T_\xi]$
 \item for any $\alpha<\c$ $\{ a_\eta:\eta<\xi\}\cap \{ d_\eta:\eta<\xi\}$,
 \item $P_\xi\subseteq S_\xi$ and $P_\xi$ is binary tree,
 \item for any $\eta<\xi$ $P_\eta\cap \{ a_\beta:\beta<\xi\}=\emptyset$.
\end{enumerate}
Now in $\alpha$-th step construction let us consider $T_\alpha\subseteq T$ and $S_\alpha$. Then let us consider a perfect set $P$ contained in $[S_\alpha]$ generated by a binary tree $P_\alpha$ such that $P\cap \{ a_\xi:\xi<\alpha\}=\emptyset$ which is possible because any perfect can be partitioned onto $\c$ many perfect sets. Then choose $a_\alpha\in [T_\alpha]\setminus (\bigcup_{\xi<\alpha} P_\xi\cup \{ d_\xi:\xi<\alpha\})$ what is guaranteed by the fact that $\b=\c$ and $a_\alpha$ is in unbounded in family $\{ f_\xi:\xi<\alpha\}$ where
$$
f_\xi(n)= \max\{ s(n):s\in P_\xi\land n\in dom(s)\}
$$
The chossing $d_\alpha\in [T_\alpha]\setminus \{ a_\xi:\xi\le\alpha\}$ finishes this step of construction.

Finally the $\{ a_\xi:\xi<\c\}$ witness the required set.

The proof of the last sentence is similar to previous one when we replace the family of binary $\{ P_\xi<\c\}$ trees by Miller trees such that every odd node has finite splitting.
\end{proof}

\section{Nonmeasurable and completely nonmeasurable unions}
It is known that for any $\sigma$-ideal
$\ci$ with a Borel base  if  $\ca$ is point-finite and its union is not in $\ci$ then  there exists a subfamily $\ca'$ of $\ca$ such that the union of its sets is not in the $\sigma$-algebra generated by the Borel sets and $\ci$ (see \cite{Buk}, \cite{BCGR}).

It is known that within ZFC it is not possible to replace the as-
sumption that the family $\ca$ is point-finite even by the one saying that
$\ca$ is point-countable (see \cite{Fre}).

In various cases it is possible to obtain more than nonmeasurability
of the union of a subfamily of $\ca$. Namely, the intersection of this union
with any measurable set that is not in $\ci$
is nonmeasurable (recall, the measurability is understood here in the sense of belonging to the $\sigma$-algebra generated by the family of Borel sets and $\ci$). Such strong
conclusion can be obtained for the ideal of first Baire category sets
under the assumption that $\ca$ is a partition, but without assuming anything about the regularity of the elements of $\ca$ (see \cite{5P}).

Some related topics are presented in \cite[chapter 14]{Kha}. Namely, the problem of the existance of the family of first category sets whose union do not posses Baire property is discussed for general topological spaces of second Baire category.

In \cite{FT}, the problem concerning null sets is discussed. It is shown that for every partition of unit inerval into Lebesgue null sets and for every $\varepsilon>0$ we can find a subfamily such that its union has inner measure smaller than $\varepsilon$ and outher measure grater than $1-\varepsilon.$  

In paper \cite{Z} it was shown how to obtain complete nonmeasurability of
the union of a subfamily of $\ca$ assuming that $\ca$ is point-finite family.
However, the result requires some  set-theoretic assumptions. Namely, we need to assume that
there is no quasi-measurable cardinal smaller than $2^\omega.$ (Recall that
$\kappa$ is quasi-measurable if there exists a $\kappa$-additive ideal
$\ci$ of subsets of $\kappa$ such that the Boolean algebra $P(\kappa)/\ci$ satisfies countable chain condition.) By the Ulam theorem (see \cite{Jech}) every quasi-measurable
cardinal is weakly inaccessible, so it is a large cardinal. 

The above result was strenghtened in paper \cite{RZ1} where it was shown that it is enough to assume that  there is no quasi-measurable cardinal not greater than $2^\omega.$

The problem concerning finding completely $\ci$-nonmeasurable sets were also discussed in papers \cite{R}, \cite{RZ2}. In those results the starting families fulfills some additional conditions.   

The aim of this section is to  discus the following problem. Let $\ci$ be a $\sigma$-ideal of subsets of $\bbr$.
Assume that  $\cp\subseteq\ci$. Is it possible that
for all $\ca\subseteq\cp$
$$
\bigcup\ca \text{ is } \ci\text{-nonmeasurable }
$$
$$
\Downarrow
$$
$$
\bigcup\ca \text{ is completely} \ci\text{-nonmeasurable}?
$$
We will consider situations when $\cp$ is a partition of $\bbr$, $\cp$ is point-finite family and $\cp$ is point-countable family. 

Throught this section $\ci$ will denote a $\sigma$-ideal of subsets of $\bbr$ satisfying the following conditions
\begin{enumerate}
 \item  $\ci $ contain singletons, i.e. $[\bbr]^\w\subseteq\ci,$
 \item $\ci$ has Borel base, i.e. $(\forall I \in\ci)(\exists B \in \Borel \cap \ci)(I\subseteq B),$
 \item $\ci$ is {\bf translation invariant}, i.e. 
 $$(\forall I \in\ci)(\forall x \in\bbr)(x + I = \{x + i :\ i\in I \} \in\ci).$$
\end{enumerate}

\begin{definition}
 Let $A\subseteq\bbr$. We say that 
 \begin{enumerate}
  \item $A$ is {\bf $\ci$-nonmeasurable} if $A$ does not belong to the $\sigma$-algebra generated by Borel sets and $\sigma$-ideal $\ci;$
  \item $A$ is {\bf completely $\ci$-nonmeasurable} if $A\cap B$ is $\ci$-nonmeasurable for every Borel set $B$ which does not belong to $\ci.$
 \end{enumerate}
\end{definition}

Let us remark that the folowing conditions are all equivalent:
\begin{enumerate}
 \item $A$ is completely $\ci$-nonmeasurable,
 \item $A\cap B$ and $A\cap (\bbr\setminus B)$ does not belong to $\ci$  for every Borel set $B$ such that $B, \bbr\setminus B\notin\ci,$
 \item $A$ intersects every Borel set which does not belong to $\ci$ and does not contain any of such sets.
\end{enumerate}
 
Let us notice that if $\ci$ is the ideal of countable sets then $A$ is completely $\ci$-nonmeasurable if and only if $A$ is a Bernstein set. That is why completely $\ci$-nonmeasurable sets are sometimes called $\ci$-Bernstein sets.

If $\ci$ is the ideal of Lebesgue null sets then $A$ is completely $\ci$-nonmeasurable if and only if its inner measure is zero and the inner measure of its complement is also zero.

We divide results into three groups: the first - ideals with Steinhaus property, the second - families consisting of finite sets and the third - families consisting of countable sets.
\subsection{Ideals with weaker Smital property}
In this subsection we will consider $\sigma$-ideals possesing weaker Smital property. This notion was introduced in \cite{lodz} and was invastigated in \cite{MZ}.
Let us recall the definition.
\begin{definition}
We say that $\ci $ has {\bf weaker Smital property} if there exists a countable dense set $D$ such that
$$
 (\forall A\in\Borel\setminus\ci)((A+D)^c\in\ci).
$$
We say that $D$ witnesses that $\ci$ has the weaker Smital property.
\end{definition}

Let us notice that the weaker Smital property is implied by Smital property and Smital property is implied by Steinhaus property.

Let us remark that the ideal $\cn$ of null subsets of $\bbr$ and the ideal $\cm$ of meager subsets of $\bbr$ have Steinhaus property.
More natural examples of ideals possesing weaker Smital property in euclidean spaces can be found in \cite{lodz}.

From the other hand, the ideal $[\bbr]^{\le\w}$ of countable subsets of reals does not have weaker Smital property.

\begin{theorem}
Assume $\ci$ has weaker Smital property. Then there exists a partition
$\cp\subseteq\ci$ of $\bbr$ such that for every $\ca\subseteq\cp$
$$
\bigcup\ca \text{ is } \ci\text{-nonmeasurable }
$$ 
$$
\Downarrow
$$
$$
\bigcup\ca \text{ is completely } \ci\text{-nonmeasurable.}
$$
\end{theorem}
\begin{proof}
Let $D$ be a set witnessing that $\ci $ has weaker Smital propoerty. We can assume that $D$ is a subgroup of $(\mathbb{R},+)$
For $x, y \in\bbr$ let $x \sim y \leftrightarrow x-y\in D.$ Set
$$
\cp = \bbr/\sim = \{x_\alpha + D :\ \alpha\in 2^\w \}.
$$
Take $\ca\subseteq\cp$ such that $\bigcup\ca$ is  $\ci$-nonmeasurable.
Assume that $\bigcup\ca$ is not completely $\ci$-nonmeasurable.
Then $\bigcup\ca\notin\ci$ and $\bigcup(\cp\setminus\ca)\notin\ci$ and at least one of this sets
contains $\ci$-possitive Borel set. Without loss of generality we can assume that $\bigcup\ca$ contains $\ci$-possitive Borel set.
By weaker Smital property of $\ci$, the set $(\bigcup\ca +D)^c$ belongs to $\ci$. Notice that 
$$
\bigcup\ca +D =\bigcup\ca\text{ and }\bigcup\ca\cap\bigcup(\cp\setminus\ca)=\emptyset
$$
Contradiction.
\end{proof}

\subsection{Finite sets}
In this subsection we will deal with families consisting of finite sets.
\begin{theorem}
Let $\cp \subseteq [\bbr]^{<\w}$ be a partition of $\bbr.$ Then
\begin{enumerate}
 \item there is $\ca_0\subseteq \cp$ such that $\bigcup\ca_0$ is completely $\ci$-nonmeasurable;
 \item there is $\ca_1\subseteq \cp$ such that $\bigcup\ca_1$ is $\ci$-nonmeasurable but is not
completely $\ci$-nonmeasurable.
\end{enumerate}
\end{theorem}
\begin{proof}
Family $\ca_0$ can be constructed in the standard way following construction of Bernstein set.

To prove the second part let us enumerate 
$$
\cp = \{Y_\alpha : \alpha\in2^\w \},
$$
$$
Y_\alpha = \{y_0^\alpha , y_1^\alpha , \ldots , y_n^\alpha \},\quad y_0^\alpha<y_1^\alpha< \ldots  <y_n^\alpha.
$$

Define $X_k = \{y_k^\alpha : \alpha\in 2^\w \}.$ Witout lost of generality $X_0\notin\ci.$
We can find $r\in\bbr$ such that $X_0 \cap (-\infty, r ) \notin\ci$ and
$X_0 \cap (r , +\infty) \notin\ci.$

Set $\cp^+ = \{Y_\alpha :\ y_0^\alpha > r \} \subseteq\cp.$
We have that $\bigcup\cp^+ \subseteq (r , +\infty)$ and $\bigcup\cp\notin\ci.$
Find $\ca_1 \subseteq\cp^+$ such that $\bigcup\ca_1$ is $\ci$-nonmeasurable. 
$\bigcup\ca_1$ is not
completely $\ci$-nonmeasurable.
\end{proof}

\subsection{Countable sets}
In this subsection we will deal with families consisting of countable sets.
\begin{theorem} 
Assume that $\cp\subseteq [\bbr]^{\le\w}$ is a point-countable cover of $\bbr.$ Then we can find
$\ca\subseteq\cp$ such that $\bigcup\ca$ is completely $[\bbr]^{\le\w}$-nonmeasurable.
\end{theorem}
\begin{proof}
We will slightly modify the standard construction of Bernstein set. Let 
$$
\{Q_\alpha:\ \alpha<2^\w\}
$$
be enumeration of all nonempty perfect subsets of $\bbr$. By transfinite induction on $\alpha<2^\w$ we will construct 
$$
A_\alpha\in\cp,\quad x_\alpha\in Q_\alpha
$$
satisfying the following conditions
\begin{enumerate}
 \item $A_\alpha\cap Q_\alpha\neq\emptyset,$
 \item $A_\alpha\cap\{x_\beta: \beta<\alpha\}=\emptyset,$
 \item $x_\alpha\notin\bigcup_{\beta\le\alpha}A_\beta.$
\end{enumerate}
The construction can be made because at $\alpha$-step
$$
Q_\alpha\setminus\left(\bigcup\{A\in\cp:\ \exists\beta<\alpha\ x_\beta\in A\}\cup\bigcup_{\beta<\alpha}A_\beta\right)\neq\emptyset.
$$
So, we can find $A_\alpha$ and $x_\alpha$ fulfilling our requirements.

At the end, we get $\ca=\{A_\alpha:\ \alpha<2^\w\}$ such that $\bigcup\ca\cap Q_\alpha\neq\emptyset$ and $\bigcup\ca\cap\{x_\alpha:\ \alpha< 2^\w\}=\emptyset$, what shows that $\bigcup\ca$ is completely $[\bbr]^{\le\w}$-nonmeasurable. 
\end{proof}

\begin{theorem}[$\neg CH$]\label{nch}
Assume that $\cp\subseteq [\bbr]^{\le\w}$ is a partition of $\bbr.$ Then we can find
$\ca\subseteq\cp$ such that $\bigcup\ca$ is $[\bbr]^{\le\w}$-nonmeasurable but is not completely
$[\bbr]^{\le\w}$-nonmeasurable.
\end{theorem}
\begin{proof}
Take $\ca\subseteq\cp$ such that $|\ca| = \w_1.$
$|\bigcup\ca| = \w_1 < 2^\w.$ So, $\bigcup\ca$ is $[\bbr]^{\le\w}$-nonmeasurable.
Fix $\{Q_\alpha : \alpha\in 2^\w \}$ a family of pairwise disjoint perfect sets.
There exists $\alpha$ such that $Q_\alpha \cap \bigcup\ca = \emptyset.$ 
So, $\bigcup\ca$ is not completely
$[\bbr]^{\le\w}$-nonmeasurable.
\end{proof}

\begin{theorem}[$CH$]\label{ch}
There is $\cp\subseteq [\bbr]^{\le\w}$ a partition of $\bbr$ such that for any $\ca\subseteq\cp$
$$ \bigcup\ca \text{ is } [\bbr]^{\le\w} \text{-nonmeasurable }
$$
$$
\Downarrow
$$
$$
\bigcup\ca \text{ is completely } [\bbr]^{\le\w} \text{-nonmeasurable. }$$
\end{theorem}
\begin{proof}
Let $\{Q_\alpha : \alpha\in\w_1\}$ be an enumeration of all perfect subsets of $\bbr.$
We can construct a partition $\cp = \{X_\alpha : \alpha\in\w_1 \}\subseteq [\bbr]^{\le\w}$ 
in such a way that
$X_\alpha \cap Q_\beta \neq \emptyset$ for every $\beta <\alpha.$
Now, take $\ca \subseteq\cp$ such that $|\ca| = |\cp\setminus\ca | = \w_1.$ Then
$\bigcup\ca \cap Q_\alpha \neq\emptyset$ and
$\bigcup(\cp \setminus \ca) \cap Q_\alpha \neq \emptyset$ for every $\alpha < \w_1.$
So, $\bigcup\ca$ is completely $[\bbr]^{\le\w}$-nonmeasurable.
\end{proof}

As a consequence of Theorem \ref{nch} and Theorem \ref{ch} we get the following characterisation of Continuum Hypothesis.
\begin{cor}
The following statements are equivalent:
\begin{enumerate}
 \item CH,
 \item there is $\cp\subseteq [\bbr]^{\le\w}$ a partition of $\bbr$ such that for any $\ca\subseteq\cp$
$$
\bigcup\ca \text{ is } [\bbr]^{\le\w}\text{-nonmeasurable }
$$
$$
\Updownarrow
$$
$$
\bigcup\ca \text{ is completely } [\bbr]^{\le\w}\text{-nonmeasurable.}
$$
\end{enumerate}
\end{cor}

%

\end{document}